\documentclass[a4paper,12pt]{article}

\usepackage{amsmath,amssymb,amsthm,mathtools}
\usepackage{geometry}
\usepackage{booktabs}
\usepackage{enumitem}
\usepackage{hyperref}
\usepackage{xcolor}

\hypersetup{
  colorlinks=true,
  linkcolor=blue!55!black,
  citecolor=blue!55!black,
  urlcolor=blue!55!black,
  pdftitle={The SMFCE Paradigm: A Unified Modeling Framework Based on State, Cost, and Extremal Evolution},
  pdfauthor={Wei Ba, Ang Fu, Shengsong Kong, Zhongchuan Fu}
}
\setlist{itemsep=0.25em,topsep=0.4em}
\numberwithin{equation}{section}

\newtheorem{principle}{Modeling Principle}[section]

\newtheorem{lemma}[principle]{Lemma}
\newtheorem{proposition}[principle]{Proposition}

\theoremstyle{definition}
\newtheorem{definition}[principle]{Definition}
\theoremstyle{remark}
\newtheorem{remark}[principle]{Remark}

\DeclareMathOperator*{\argmin}{arg\,min}
\DeclareMathOperator{\Crit}{Crit}
\DeclareMathOperator{\Hess}{Hess}
\DeclareMathOperator{\Exp}{Exp}
\DeclareMathOperator{\Retr}{Retr}
\DeclareMathOperator{\Proj}{Proj}

\title{The SMFCE Paradigm: A Unified Modeling Framework Based on State, Cost, and Extremal Evolution}
\author{Wei Ba\textsuperscript{1}\quad Ang Fu\textsuperscript{1}\quad Shengsong Kong\textsuperscript{2}\quad Zhongchuan Fu\textsuperscript{2*}\\[0.5em]
	\small \textsuperscript{1}Independent Researcher\\[0.5em]
\small \textsuperscript{2}Harbin Institute of Technology, Harbin, Heilongjiang 150001, China\\[0.5em]
\small \textsuperscript{*}Corresponding author. E-mail: \href{mailto:fuzhongchuan@hit.edu.cn}{fuzhongchuan@hit.edu.cn}}
\date{}

\begin{document}

\maketitle

\begin{abstract}
This paper proposes a unified modeling framework centered on ``state--cost--extremal evolution,'' referred to as the SMFCE paradigm. The paradigm represents the effective state of a system as a point in a state space, characterizes state cost by a single-valued scalar field, and describes system evolution through local cost descent or local extremum selection. Its mathematical core employs classical tools from differential geometry, gradient systems, and dynamical systems. The contribution of this paper is not to reintroduce these foundational results, but to organize them into two operational modeling principles. On finite-dimensional smooth Riemannian manifolds under appropriate regularity conditions, the path independence of cost differences and the cost monotonicity of negative gradient flows follow directly from classical results, while the descent property of discrete local-minimum transitions follows directly from the definition of a minimum. For same-level iterative emergence and present-centeredness, the paper provides conditional mathematical formulations using, respectively, parameterized gradient systems and bifurcation theory, and Markov processes and receding-horizon optimization. The paper further introduces three layered metrics---cost quantity, directed coupling degree, and commitment fulfillment rate---at the state, transition, path, and system levels. Finally, the paper discusses semantic modeling of large language model training and inference within this framework, explicitly distinguishes classical mathematical results, conditional propositions, and empirical adaptation models, and provides reproducible empirical testing procedures for the two modeling principles.

\end{abstract}

\section{Introduction}

Complex systems commonly involve problems of state representation, evaluation mechanisms, local decision-making, and the formation of macroscopic structures. Without a unified set of mathematical objects, these problems can easily remain at the level of conceptual analogy. The SMFCE paradigm uses state spaces, cost scalar fields, and local extremal evolution as a common language, seeking to provide a unified modeling structure for continuous systems, discrete decision systems, and semantic processes in large language models.

The aim of this paper is not to claim that all systems inherently obey the same dynamics. Rather, it is to provide a set of testable modeling conditions: when a system admits a sufficient state representation, its cost can be represented as a state function, and its local evolution can be approximated by a cost-descent process, the SMFCE framework may be used to analyze its path properties, stability, cross-system correspondences, and multiscale behavior.

State manifolds, exact differential forms, negative gradient flows, and Lyapunov monotonicity are all classical topics in differential geometry and dynamical systems\cite{lee2013,absil2008,khalil2002}. Accordingly, this paper does not claim these mathematical objects themselves as original contributions. Its contributions are limited to the following three aspects:
\begin{enumerate}
  \item abstracting the evolution mechanisms of complex systems into two modeling principles---``sufficient state representation'' and ``cost-extremal evolution''---and providing corresponding data-driven testing procedures;
  \item providing an illustrative instantiation of the general paradigm in the semantic setting of large language models by taking semantic equivalence classes as state units, semantic entropy as the core cost component, and a training--inference dual-timescale account of semantic-field shaping and local evolution;
  \item defining three layered metrics---cost quantity, directed coupling degree, and commitment fulfillment rate---to describe state improvement, the standardized advantage of a local transition, and the extent to which actual coupling attains its reference level, respectively.
\end{enumerate}

Within the landscape of related research, three bodies of work are most closely connected to this paper. First, Riemannian manifold optimization\cite{absil2008} focuses on the design of numerical optimization algorithms on constrained manifolds, with particular attention to convergence and computational efficiency. It provides the mathematical tools for gradient evolution used in this paper, but it does not offer universal criteria for state representation or an empirical testing framework tailored to complex-system modeling. Second, the free-energy principle and active inference\cite{friston2010} use variational free energy to construct a unified model of cognition and brain function. This framework, however, is deeply tied to assumptions about Bayesian generative models, is largely limited in applicability to intelligent systems, and has not established a standardized modeling and testing workflow. Third, geometric approaches to cognition, including the recently proposed geometric theory of cognition\cite{ale2025}, apply Riemannian geometry to the interpretation of human cognitive processes. These are domain-specific models rather than a general cross-system methodology. In contrast to these bodies of work, the central contribution of this paper is neither a new mathematical tool nor a domain-specific model. It systematizes classical tools from differential geometry and dynamical systems into a general, cross-domain, and testable modeling paradigm, thereby providing a unified language and operational standards for modeling, comparing, and validating different complex systems.

The remainder of the paper is organized as follows. Section~\ref{sec:preliminaries} introduces the mathematical objects and basic notation. Section~\ref{sec:principles} presents the two modeling principles and their empirical tests. Section~\ref{sec:theorems} summarizes the classical foundational results. Section~\ref{sec:conditional} discusses two classes of conditional propositions. Section~\ref{sec:llm} presents a semantic model of large language model training and inference. Section~\ref{sec:boundary} summarizes the theoretical levels and scope of applicability. Section~\ref{sec:metrics} presents the three core metrics and their domains of definition.

\section{Mathematical Preliminaries and Basic Objects}
\label{sec:preliminaries}

\subsection{Scenario-State Space}

Let the effective state space of the system be a finite-dimensional smooth manifold $M$. A point $p\in M$ represents the complete effective state of the system at a given time. To define lengths, angles, gradients, and directions of steepest descent, equip $M$ with a Riemannian metric $g$, thereby obtaining the Riemannian manifold $(M,g)$.

If the system is subject to feasibility constraints, let
\begin{equation}
  M_{\mathrm{adm}}\subseteq M
\end{equation}
denote the admissible state domain. Depending on the system, $M_{\mathrm{adm}}$ may be a submanifold with boundary, a closed constraint set, or a stratified space.

\subsection{Cost Scalar Field}

Let a single evolution cycle be $I=[t_0,t_0+\Delta T]$. Define the cost scalar field by
\begin{equation}
  C:M\times I\longrightarrow\mathbb{R}.
\end{equation}
For a fixed time $t\in I$, write
\begin{equation}
  C_t:=C(\cdot,t):M\longrightarrow\mathbb{R}.
\end{equation}
$C_t(p)$ denotes the cost of state $p$ at time $t$. The cost may correspond to energy, loss, uncertainty, task residual, degree of constraint violation, or a weighted combination of these quantities.

When discussing path integrals and gradient flows, assume $C_t\in C^1(M)$. When discussing Hessians, critical-point degeneracy, and bifurcations, further assume $C_t\in C^2(M)$.

\subsection{Continuous Trajectories and Discrete Neighborhoods}

A continuous evolution trajectory is denoted by
\begin{equation}
  \gamma:I\longrightarrow M,
\end{equation}
and a discrete evolution sequence by $p_0,p_1,p_2,\ldots\in M$. For a discrete state $p_j$, let $N(p_j)$ denote its locally reachable neighborhood. When an exact local-minimum rule is used, assume that $N(p_j)\cap M_{\mathrm{adm}}$ is finite, or that this set is compact and $C$ is continuous on it.

\section{Two Foundational Modeling Principles}
\label{sec:principles}

\begin{principle}[Scenario-State and Positional Uniqueness]
Within a single evolution cycle $I$, the effective scenario-state of the system can be represented by a position in the state space $M$. For every fixed $(p,t)\in M\times I$, the cost value $C(p,t)\in\mathbb{R}$ is uniquely determined. Equivalently, for every fixed time $t$, $C_t$ is a single-valued state function defined on $M$.
\end{principle}

This principle requires the state variables to contain all effective information needed to determine the cost. If the same apparent state at the same time corresponds to different costs, then the current state representation omits historical, environmental, or latent variables, and the state space should be augmented. Formally, introduce a raw history space $\mathcal H$ and a state representation $\pi:\mathcal H\to M$, and require that there exist a function $C$ such that
\begin{equation}
  C_{\mathrm{raw}}(h,t)=C(\pi(h),t),
  \qquad h\in\mathcal H.
\end{equation}

\paragraph{Empirical test of state-representation sufficiency}
When the cost can only be obtained through a noisy observation $Y$, Principle~3.1 can be operationalized as the null hypothesis
\begin{equation}
  H_0^{(1)}:\quad
  Y\perp\!\!\!\perp H\mid\bigl(\pi(H),t\bigr),
  \label{eq:state-sufficiency-null}
\end{equation}
meaning that, conditional on the state representation and time, the complete history provides no additional information about the observed cost. In practice, $Y$ should be an independent observation or a predefined empirical cost, and the test should proceed as follows:
\begin{enumerate}
  \item On the training sample, fit a baseline model $m_0$ using only $(\pi(H),t)$ and an extended model $m_1$ that additionally uses candidate historical, environmental, or latent variables $Z$.
  \item On a held-out sample, compute the predictive risks $R_0$ and $R_1$, and assess $\Delta R=R_0-R_1$ using cross-validation, permutation tests, or conditional-independence tests.
  \item If $\Delta R$ is significantly greater than a prespecified tolerance, or if the residuals of $m_0$ remain significantly associated with $Z$, reject the sufficiency of the current state representation and augment $M$. Otherwise, report confidence intervals and an identifiable error bound, without interpreting ``failure to reject'' as a rigorous proof.
\end{enumerate}
If $C$ itself is a proxy defined from $p=\pi(h)$, then equation~\eqref{eq:state-sufficiency-null} degenerates into a tautology by construction and cannot establish the sufficiency of the state representation. The test target must include an observed outcome or cost measure independent of that definition.

The cost field is allowed to vary slowly during a cycle. On a compact relevant region $K\subseteq M$ containing the trajectory, define
\begin{equation}
  \varepsilon_I
  =\sup_{(p,t)\in K\times I}
  \left|\frac{\partial C(p,t)}{\partial t}\right|.
\end{equation}
The fundamental theorem of calculus gives
\begin{equation}
  |C(p,t)-C(p,t_0)|
  \leq \varepsilon_I|t-t_0|.
\end{equation}
Given a tolerance $\eta>0$, if $\varepsilon_I\Delta T\leq\eta$, then the frozen field $C_{t_0}$ may be used to approximate $C_t$ over a single cycle.

If approximate stability of the descent direction is also required, one may further assume
\begin{equation}
  \sup_{p\in K}
  \left\|\nabla_g C(p,t)-\nabla_g C(p,t_0)\right\|_g
  \leq\delta_I.
\end{equation}

\begin{principle}[Cost-Extremal Evolution]
In the absence of strong external intervention, under a stable constraint structure and with a cost field that is approximately static over a single cycle, the system evolves according to a local cost-extremum rule. In the continuous setting, this is represented by motion along feasible negative-gradient directions; in the discrete setting, it is represented by selecting lower-cost states from a locally reachable neighborhood. This principle specifies only a local evolution rule and does not presuppose a globally optimal terminal state.
\end{principle}

In the unconstrained setting, the continuous dynamics associated with the frozen field $C:=C_{t_0}$ are
\begin{equation}
  \dot\gamma(t)=-k\nabla_g C(\gamma(t)),
  \qquad k>0,
  \label{eq:gradient-flow}
\end{equation}
where the Riemannian gradient is defined by
\begin{equation}
  g_p(\nabla_g C(p),v)=dC_p(v),
  \qquad v\in T_pM.
\end{equation}
If the vector field $-\nabla_g C$ is locally Lipschitz, then a Picard--Lindel\"of-type existence-and-uniqueness result implies that, for any given initial state, equation~\eqref{eq:gradient-flow} has a unique local solution\cite{coddington1955}.

When the system is constrained to $M_{\mathrm{adm}}$, the projected dynamics should be written as
\begin{equation}
  \dot\gamma(t)
  =\Proj_{T_{M_{\mathrm{adm}}}(\gamma(t))}
  \left[-k\nabla_g C(\gamma(t))\right],
  \label{eq:projected-flow}
\end{equation}
where $T_{M_{\mathrm{adm}}}(p)$ is the feasible tangent cone. If the feasible set is a smooth submanifold without boundary, its tangent cone reduces to a linear tangent space. For a nonsmooth closed constraint set, one may instead use, under appropriate regularity conditions,
\begin{equation}
  \dot\gamma(t)
  \in-k\nabla_g C(\gamma(t))-N_{M_{\mathrm{adm}}}(\gamma(t)),
  \qquad \gamma(t)\in M_{\mathrm{adm}},
\end{equation}
where $N_{M_{\mathrm{adm}}}$ denotes the corresponding normal cone. These are standard modeling forms for projected dynamical systems and differential inclusions\cite{dupuis1993,aubin1984}.

The discrete local-extremum rule is
\begin{equation}
  p_{j+1}\in
  \argmin_{p\in N(p_j)\cap M_{\mathrm{adm}}}C(p).
  \label{eq:discrete-argmin}
\end{equation}
Another class of discrete models is Riemannian gradient descent,
\begin{equation}
  p_{j+1}=\Exp_{p_j}\left(-\eta_j\nabla_gC(p_j)\right),
\end{equation}
or the exponential map may be replaced by a retraction $\Retr$. Descent for this model additionally requires gradient smoothness and suitable step-size conditions; it cannot be identified unconditionally with equation~\eqref{eq:discrete-argmin}.

\paragraph{Trajectory test of the gradient-evolution hypothesis}
Given $C$ and $g$ specified independently or estimated using only a training set, estimate velocities $\widehat v_i$ from trajectory observations $(p_i,t_i)$. In the unconstrained case, minimize over $k\geq0$ the normalized residual
\begin{equation}
  \mathcal R_{\mathrm{grad}}
  =
  \frac{\sum_i w_i
  \left\|\widehat v_i+k\nabla_gC(p_i,t_i)\right\|_g^2}
  {\sum_i w_i\|\widehat v_i\|_g^2},
  \label{eq:gradient-residual-test}
\end{equation}
and also report the directional cosine between $\widehat v_i$ and $-\nabla_gC(p_i,t_i)$, the proportion of actual cost-decreasing transitions, and their confidence intervals. In the constrained case, replace $-\nabla_gC$ with the projected vector field in equation~\eqref{eq:projected-flow}. In the discrete case, test the frequency of transitions to lower-cost states, the likelihood relative to an $\argmin$ model or a random-choice model, and predictive error on held-out trajectories. Principle~3.2 should be regarded as empirically supported only when these measures meet prespecified tolerances and exhibit a stable advantage over time-varying, stochastic, or nongradient alternatives. If $C$, $g$, and $k$ are all freely fitted using the same trajectories, sample splitting and complexity control are required to prevent the gradient-flow hypothesis from becoming unfalsifiable through overfitting.

\section{Foundational Mathematical Results}
\label{sec:theorems}

\begin{proposition}[Classical result: path independence of cost differences]
Fix $t\in I$. If $C_t\in C^1(M)$ is a globally single-valued function, then for every piecewise smooth path $\gamma:[a,b]\to M$,
\begin{equation}
  \int_\gamma dC_t
  =C_t(\gamma(b))-C_t(\gamma(a)).
\end{equation}
In particular, for every closed path $\Gamma$, $\oint_\Gamma dC_t=0$.
This is a direct application of the fundamental theorem for line integrals of exact $1$-forms on smooth manifolds\cite{lee2013}. A proof is included for completeness.
\end{proposition}

\begin{proof}
By the definition of a line integral and the chain rule,
\begin{align}
  \int_\gamma dC_t
  &=\int_a^b dC_t|_{\gamma(s)}\bigl(\dot\gamma(s)\bigr)\,ds\\
  &=\int_a^b\frac{d}{ds}C_t(\gamma(s))\,ds\\
  &=C_t(\gamma(b))-C_t(\gamma(a)).
\end{align}
When $\Gamma(a)=\Gamma(b)$, the right-hand side is zero.
\end{proof}

\begin{remark}
If only a closed $1$-form $\omega$ is given, without a global potential function, then $\omega$ can in general be written as $dC_U$ only locally. Deducing global path independence from local closedness requires appropriate topological conditions\cite{lee2013}.
\end{remark}

\begin{lemma}[Classical result: monotonicity under a frozen cost field]
Let $C\in C^1(M)$ and let $\gamma$ satisfy equation~\eqref{eq:gradient-flow}. Then
\begin{equation}
  \frac{d}{dt}C(\gamma(t))
  =-k\left\|\nabla_gC(\gamma(t))\right\|_g^2
  \leq0.
  \label{eq:lyapunov-descent}
\end{equation}
Therefore, $C$ is a Lyapunov function for the negative gradient flow.
This is the standard Lyapunov property of gradient systems\cite{absil2008,khalil2002}.
\end{lemma}

\begin{proof}
By the chain rule, the definition of the Riemannian gradient, and equation~\eqref{eq:gradient-flow},
\begin{align}
  \frac{d}{dt}C(\gamma(t))
  &=dC_{\gamma(t)}\bigl(\dot\gamma(t)\bigr)\\
  &=g_{\gamma(t)}\left(\nabla_gC(\gamma(t)),
    -k\nabla_gC(\gamma(t))\right)\\
  &=-k\left\|\nabla_gC(\gamma(t))\right\|_g^2\leq0.
\end{align}
\end{proof}

\begin{proposition}[Descent condition for a time-varying field]
If $C=C(p,t)$ and $\dot\gamma=-k\nabla_gC(\gamma(t),t)$, then
\begin{equation}
  \frac{d}{dt}C(\gamma(t),t)
  =\frac{\partial C}{\partial t}(\gamma(t),t)
  -k\|\nabla_gC(\gamma(t),t)\|_g^2.
\end{equation}
Therefore, a sufficient condition for the actual time-varying cost to remain nonincreasing is
\begin{equation}
  \frac{\partial C}{\partial t}
  \leq k\|\nabla_gC\|_g^2
\end{equation}
along the trajectory.
\end{proposition}

\begin{remark}[Descent of discrete local-minimum transitions]
Suppose the minimum in equation~\eqref{eq:discrete-argmin} exists and $p_j\in N(p_j)\cap M_{\mathrm{adm}}$. Then
\begin{equation}
  C(p_{j+1})\leq C(p_j).
\end{equation}
Because $p_j$ itself belongs to the feasible candidate set and $p_{j+1}$ is a minimizer over that set, the conclusion follows directly from the definition of a minimum.
\end{remark}

Define the critical set by
\begin{equation}
  \Crit(C)=\{p\in M:\nabla_gC(p)=0\}.
\end{equation}
If the forward trajectory of the negative gradient flow exists for all $t\geq0$ and remains in a compact sublevel set, LaSalle's invariance principle implies that its $\omega$-limit set is contained in $\Crit(C)$\cite{lasalle1960}. If the critical points are isolated or a condition such as the \L{}ojasiewicz gradient inequality holds\cite{lojasiewicz1984}, convergence to a single point may be obtained under the corresponding assumptions. Monotone cost descent alone does not guarantee that a limit point is a local minimum; the type of critical point must still be determined by Hessian and stability analyses\cite{absil2008,khalil2002}.

\section{Two Classes of Conditional Derived Propositions}
\label{sec:conditional}

\subsection{Same-Level Iterative Emergence}

Here ``same-level'' means that state evolution and bifurcation occur within a fixed effective state space $(M,g)$ and its family of parameterized cost fields. The parameter $\mu$ may change the cost landscape, critical points, stable branches, and basins of attraction, but equation~\eqref{eq:slow-fast} always assumes a fixed $(M,g)$.

Let the parameterized family of cost fields be $C_\mu:M\to\mathbb{R}$, and consider the slow--fast system
\begin{equation}
  \dot\gamma(t)=-k\nabla_gC_{\mu(t)}(\gamma(t)),
  \qquad
  \dot\mu(t)=\epsilon r(\mu(t)),
  \qquad 0<\epsilon\ll1.
  \label{eq:slow-fast}
\end{equation}

\begin{proposition}[Candidate bifurcation condition for same-level emergence]
If an equilibrium branch undergoes a local bifurcation near $(p^*,\mu^*)$ triggered by a zero eigenvalue, then
\begin{equation}
  \nabla_gC_{\mu^*}(p^*)=0,
  \qquad
  \det\left(\Hess_gC_{\mu^*}(p^*)\right)=0.
  \label{eq:bifurcation-candidate}
\end{equation}
Equation~\eqref{eq:bifurcation-candidate} is only a candidate degeneracy condition. Whether a bifurcation actually occurs, and its type, further depends on transversality, higher-order nondegeneracy conditions, and center-manifold or normal-form analysis\cite{kuznetsov2004}.
\end{proposition}

If a new asymptotically stable equilibrium $p_{\mathrm{new}}$ appears after the bifurcation and $\gamma(t_0)$ lies in its basin of attraction $\mathcal B(p_{\mathrm{new}})$, then
\begin{equation}
  \lim_{t\to\infty}\gamma(t)=p_{\mathrm{new}}.
\end{equation}
This provides a mathematical prototype for the formation of a new stable branch, attractor, or observable ordered structure within a fixed state space. Emergence cannot be declared inevitable on the basis of Hessian degeneracy alone.

\subsection{Present-Centeredness and the Limits of Long-Term Prediction}

The local evolution of the system may be written in continuous form as
\begin{equation}
  \dot\gamma(t)=F(\gamma(t),t),
\end{equation}
or in discrete form as
\begin{equation}
  p_{j+1}=\Phi(p_j),
  \qquad\text{or}\qquad
  p_{j+1}\sim P(\cdot\mid p_j).
\end{equation}
When $p_j$ contains all the information needed to determine the conditional distribution of the next step, the process has a Markov representation\cite{norris1997}. If the state omits relevant historical variables, it can only be regarded as approximately Markovian.

\begin{proposition}[Prediction boundary under sensitive dependence]
If a discrete dynamical system, or the corresponding random dynamical system, has a maximal Lyapunov exponent $\lambda_{\max}>0$, then a sufficiently small initial perturbation may approximately satisfy, over a finite regime,
\begin{equation}
  \|\delta p_n\|
  \approx\|\delta p_0\|e^{\lambda_{\max}n},
\end{equation}
thereby limiting the accuracy of long-term prediction\cite{eckmann1985}. A positive Lyapunov exponent does not follow automatically from a negative gradient flow; it must be established separately through a time-varying field, stochastic perturbations, multi-agent coupling, nonsmooth switching, or some other mechanism.
\end{proposition}

Under incomplete information, computational limitations, or a continuously changing environment, present-centeredness can be represented by receding-horizon optimization with state-transition constraints:
\begin{align}
  \min_{a_j,\ldots,a_{j+H}}\quad
  &\sum_{r=j}^{j+H}C(p_r,a_r),\\
  \text{s.t.}\quad
  &p_{r+1}=F(p_r,a_r),\qquad
  p_r\in M_{\mathrm{adm}}.
\end{align}
This model follows the standard receding-horizon logic of model predictive control\cite{rawlings2017}. It emphasizes verifiable decisions over a finite window using currently available information, without claiming that local greediness is equivalent to global optimality.

\section{Semantic Modeling for LLM Training and Inference}
\label{sec:llm}

This section provides an illustrative instantiation of the SMFCE modeling paradigm developed above in the context of large language models. The basic adaptation conventions are as follows: semantic equivalence classes are taken as the basic units of semantic states, each corresponding to a single point on the manifold; semantic entropy serves as the core cost scalar for constructing a single-valued scalar field over the semantic state space\cite{kuhn2023,farquhar2024}; and the semantic generation process corresponds to local cost-minimizing evolution on the manifold. A complete system of quantitative measures, theoretical derivations, and empirical tests will be developed systematically in future dedicated work.

\subsection{Token-Level Cost}

Given a context $x$ and a candidate token $y$, define
\begin{equation}
  C_{\mathrm{token}}(y\mid x)
  =-\log p_\theta(y\mid x).
\end{equation}
The uncertainty of the output distribution is
\begin{equation}
  H(p_\theta)
  =-\sum_y p_\theta(y\mid x)
  \log p_\theta(y\mid x).
\end{equation}
Under greedy decoding, selecting the token with the highest conditional probability is equivalent to selecting the token with the smallest $C_{\mathrm{token}}$. When temperature sampling, top-$k$, or top-$p$ truncation is used, the next token should be described by a conditional transition distribution rather than by a deterministic local-minimum rule.

\subsection{Semantic-Level Cost Field}

Let semantic objects be represented through an embedding map
\begin{equation}
  \phi:\mathcal{X}_{\mathrm{sem}}\longrightarrow\mathbb{R}^d,
\end{equation}
and let $\mathcal{S}\subseteq\mathbb{R}^d$ be the effective semantic state set. An approximate semantic cost field may be constructed as
\begin{equation}
  C_{\mathrm{sem}}:\mathcal{S}\longrightarrow\mathbb{R},
\end{equation}
for example,
\begin{equation}
  C_{\mathrm{sem}}(s)
  =\alpha H_{\mathrm{sem}}(s)
  +\beta R_{\mathrm{task}}(s)
  +\lambda V_{\mathrm{constraint}}(s).
\end{equation}
Here $H_{\mathrm{sem}}$ denotes semantic uncertainty, $R_{\mathrm{task}}$ denotes the task residual, and $V_{\mathrm{constraint}}$ denotes the degree of constraint violation. Each component must be operationally defined through observations, estimators, or proxy indicators appropriate to the specific task.

Semantic entropy is the core foundational component of the cost field and corresponds to the intrinsic uncertainty of a semantic state\cite{kuhn2023,farquhar2024}. The semantic cost-field framework can further support quantitative methods for multiple downstream tasks, including hallucination detection, generation-quality assessment, semantic-consistency checking, and anomalous-content identification, thereby adapting the same modeling framework to multiple scenarios.

\subsection{A Dual-Timescale Interpretation of Training and Inference}

During training, variation in the parameters $\theta_t$ shapes the cost field:
\begin{equation}
  C_{\mathrm{sem}}=C_{\mathrm{sem}}(\cdot;\theta_t).
\end{equation}
During inference, local state transitions are executed while the parameters are approximately fixed, which may be idealized as
\begin{equation}
  s_{j+1}
  \in
  \argmin_{s\in N(s_j)}
  C_{\mathrm{sem}}(s;\theta_{\mathrm{fixed}}).
\end{equation}
Thus, training corresponds to the slow-timescale shaping of the cost-field structure, whereas inference corresponds to fast-timescale evolution within a given field. The semantic manifold and semantic cost field are testable approximate models, not rigorous objects automatically guaranteed by the architecture of a neural network. If inference includes online learning, tool feedback, or continuously changing context, a time-varying field $C_{\mathrm{sem}}(s,t)$ should be used, or the state variables should be augmented.

From the perspective of the SMFCE paradigm, training and inference share the structural form of ``state--cost--local extremal evolution'' and can therefore be described within a unified language, but they are not strictly isomorphic. The state of training lies primarily in parameter or model space, and its evolution changes the cost field itself; the state of inference lies primarily in semantic state space, and its evolution occurs within a given cost field while the parameters remain approximately fixed. More precisely, they are two instantiations of the same paradigm on different state spaces and timescales. An approximate SMFCE structural homomorphism could be discussed only after a mapping between training states and inference states has been specified and shown to preserve cost relations and evolution directions within prescribed error bounds. This paper does not claim a strict isomorphism between training and inference.

This dual-timescale modeling framework can directly reuse the state-sufficiency and trajectory-fitting test procedures of the SMFCE paradigm, enabling quantitative falsification of the modeling assumptions and validation of their effectiveness.

\section{Theoretical Levels, Testability, and Scope of Applicability}
\label{sec:boundary}

\subsection{Rigorous Mathematical Results}

On a finite-dimensional smooth Riemannian manifold, with a globally single-valued cost function of the required regularity, well-defined constraints, and suitable compactness conditions, the following statements follow directly from classical results or definitions:
\begin{enumerate}
  \item the path integral of $dC$ depends only on the endpoints;
  \item the negative gradient flow of a frozen cost field makes the cost nonincreasing;
  \item an exact discrete local-minimum transition makes the cost nonincreasing;
  \item under precompactness and related conditions, the limit set of a gradient flow lies in the critical set;
\end{enumerate}

\subsection{Conditional Mathematical Results}

Same-level iterative emergence and long-term prediction boundaries require additional structure. Same-level emergence depends on a parameterized dynamical system within a fixed effective state space, critical degeneracy, and nondegenerate bifurcation conditions. Long-term sensitive dependence requires a positive Lyapunov exponent or another error-amplification mechanism.

\subsection{Empirical Adaptation Models}

Semantic manifolds and semantic cost fields for large language models generally belong to the empirical adaptation layer. Their validity requires, at a minimum, testing whether the state representation is sufficient, whether the cost function is observable and stable, whether actual transitions align with cost-descent directions, and whether the error of the one-cycle static approximation is controlled.

If these conditions fail, another model should be used: time-varying, stochastic, nonsmooth, or higher-dimensional. A fixed smooth gradient-flow interpretation should not be imposed on the system.

\section{Three Core Metrics}
\label{sec:metrics}

This section introduces three operational metrics within the preceding ``state--cost--extremal-evolution'' framework. They describe, respectively, the cost improvement of a state relative to a baseline, the standardized advantage of a selected transition within a local candidate set, and the extent to which actual coupling attains its reference level.

\subsection{Domain and Common Conditions}

Fix a time $t$, or a single evolution cycle satisfying the static approximation, and let $C_t$ denote the current cost field. For a discrete transition, define the admissible candidate set by
\begin{equation}
  A(u):=N(u)\cap M_{\mathrm{adm}}
  =\{v_1,\ldots,v_m\},
  \qquad m\ge 2.
  \label{eq:metric-domain}
\end{equation}
Throughout this section, $C_t$ takes finite real values on all evaluated states. For an infinite candidate set, the compactness and continuity conditions in Section~\ref{sec:preliminaries} ensure attainment of a minimum, while a probability measure on the candidate set specifies the corresponding mean and variance. The following definitions use a finite candidate set.

\subsection{Cost Quantity}

\begin{definition}[Cost quantity]
Given a baseline state $p_0$ and a current state $p$, define the cost quantity from $p_0$ to $p$ by
\begin{equation}
  \Delta C_t(p;p_0):=C_t(p_0)-C_t(p).
  \label{eq:cost-quantity}
\end{equation}
\end{definition}

Because $C_t$ is a finite-valued, single-valued function, $\Delta C_t$ is well defined and remains invariant when an arbitrary constant is added to $C_t$. A value $\Delta C_t>0$ represents a cost improvement relative to the baseline, $\Delta C_t=0$ represents no change in cost, and $\Delta C_t<0$ represents a cost increase.

\subsection{Directed Coupling Degree}

For the candidate costs $c_i:=C_t(v_i)$, define the population mean and standard deviation by
\begin{equation}
  \mu_t(u):=\frac{1}{m}\sum_{i=1}^{m}c_i,
  \qquad
  \sigma_t(u):=
  \left[\frac{1}{m}\sum_{i=1}^{m}
  \bigl(c_i-\mu_t(u)\bigr)^2\right]^{1/2}.
  \label{eq:candidate-statistics}
\end{equation}
Assume $\sigma_t(u)>0$ and let
\begin{equation}
  v\in\argmin_{w\in A(u)}C_t(w).
\end{equation}

\begin{definition}[Directed coupling degree]
The directed coupling degree of the selected transition $u\to v$ is
\begin{equation}
  S_t(u\to v):=
  \frac{\mu_t(u)-C_t(v)}{\sigma_t(u)}.
  \label{eq:directed-coupling}
\end{equation}
\end{definition}

$S_t$ is dimensionless. It measures how many standard deviations the cost of the selected state lies below the mean of the candidate set, thereby representing the standardized advantage of the selected transition relative to the candidate distribution. Under a positive affine transformation of candidate costs, $C\mapsto aC+b$ with $a>0$, $S_t$ remains unchanged. When the minimum is not unique, all tied minimizing states have the same value of $S_t$, and the realized transition is determined by a prespecified tie-breaking rule.

\subsection{Commitment Fulfillment Rate}

To compare reference capability with actual performance, let $C_{\mathrm{pure}}$ be the reference cost field and $C_{\mathrm{actual}}$ the actual cost field. The reference field $C_{\mathrm{pure}}$ is defined by a model under ideal conditions, obtained from a calibration experiment, or given as the convergent limit of a family of cost fields $C^{(k)}$ under a prespecified topology.

For a link $e=(u,v)$ that is explicitly matched across the two fields, write
\begin{equation}
  P_e:=S_{\mathrm{pure}}(e),
  \qquad
  A_e:=S_{\mathrm{actual}}(e).
\end{equation}
The single-link commitment fulfillment rate is defined when the candidate states and the link are consistently matched across the two fields, both standard deviations are positive, and $P_e>0$.

\begin{definition}[Single-link commitment fulfillment rate]
\begin{equation}
  F(e):=\frac{A_e}{P_e}.
  \label{eq:edge-fulfillment}
\end{equation}
\end{definition}

Under the calibration convention $0\le A_e\le P_e$, one has $F(e)\in[0,1]$. The equality $F(e)=1$ indicates that the actual coupling degree reaches the reference value, and a larger $F(e)$ represents a higher realization of the reference coupling level on that link.

\begin{definition}[Path-level commitment fulfillment rate]
For a path $\Gamma=(e_1,\ldots,e_n)$ consisting of $n$ aligned links, assume $\sum_{i=1}^{n}P_{e_i}^{2}>0$ and define
\begin{equation}
  F(\Gamma):=
  \frac{\sum_{i=1}^{n}P_{e_i}A_{e_i}}
       {\sum_{i=1}^{n}P_{e_i}^{2}}
  =
  \frac{\sum_{i=1}^{n}P_{e_i}^{2}F(e_i)}
       {\sum_{i=1}^{n}P_{e_i}^{2}}.
  \label{eq:path-fulfillment}
\end{equation}
\end{definition}

Equation~\eqref{eq:path-fulfillment} aggregates the link-level rates with weights $P_{e_i}^{2}$, thereby assigning greater influence to links with higher reference coupling degrees.

\begin{definition}[System-level commitment fulfillment rate]
Given a probability measure $\mathbb P$ on the path space and assuming that $F(\Gamma)$ is integrable with respect to $\mathbb P$, define
\begin{equation}
  F_{\mathrm{system}}
  :=\mathbb E_{\Gamma\sim\mathbb P}[F(\Gamma)].
  \label{eq:system-fulfillment}
\end{equation}
\end{definition}
For a finite test set $\{\Gamma_1,\ldots,\Gamma_K\}$ with uniform weights, the empirical estimator is
\begin{equation}
  \widehat F_{\mathrm{system}}
  :=\frac{1}{K}\sum_{k=1}^{K}F(\Gamma_k).
  \label{eq:empirical-system-fulfillment}
\end{equation}

\section{Conclusion}

This paper formulates SMFCE as a state--cost--extremal-evolution framework grounded in two modeling principles. The principle of scenario-state and positional uniqueness specifies the modeling conditions for state representations and single-valued cost fields, while the principle of cost-extremal evolution specifies the local evolution mechanism of a system. On this basis, path independence and cost monotonicity under a frozen field are supplied by classical mathematical results, while descent under a discrete local-minimum transition follows directly from the definition of a minimum. These results form the mathematical foundation of the framework but are not claimed as original theorems of this paper. Same-level iterative emergence describes the formation of new stable branches, attractors, or ordered structures within a fixed effective state space, while present-centeredness describes the boundary of local decision-making.

Large language model training and inference can be given a dual-timescale semantic interpretation within this framework, but the semantic state space and semantic cost field must still be established through experimental operationalization and error testing.

Cost quantity, directed coupling degree, and commitment fulfillment rate provide layered metrics at the state, transition, path, and system levels.

The theoretical value of SMFCE therefore lies not in reducing all complex systems to unconditional gradient flows, but in providing a unified modeling language with explicit theoretical levels, inspectable conditions, and measurable errors.

\section*{Author Contributions}

Wei Ba proposed the core research paradigm and theoretical framework and was responsible for the overall research design and the drafting of the main body of the paper. Ang Fu verified, supplemented, and standardized the mathematical formulas throughout the manuscript and contributed to content expansion and logical refinement. Shengsong Kong developed and validated the experimental framework for LLM-related research and assisted with proofreading the mathematical derivations throughout the manuscript. Zhongchuan Fu oversaw the academic quality of the research, provided guidance throughout the process, and assumed responsibility for submission correspondence and long-term academic communication.

\end{document}